\documentclass[10pt,oneside,leqno]{amsart}
\usepackage{amsxtra}
\usepackage{amsopn}
\usepackage{color}
\usepackage{amsmath,amsthm,amssymb}
\usepackage{amscd}
\usepackage{amsfonts}
\usepackage{latexsym}
\usepackage{verbatim}

\theoremstyle{plain}
\newtheorem{theorem}{Theorem}[section]
\newtheorem{definition}[theorem]{Definition}
\newtheorem{lemma}[theorem]{Lemma}
\newtheorem{proposition}[theorem]{Proposition}
\newtheorem{corollary}[theorem]{Corollary}
\newtheorem{remark}[theorem]{Remark}
\newtheorem{example}[theorem]{Example}
\newtheorem{question}[theorem]{QUESTION}
\newtheorem{remark-question}[section]{Remark-Question}
\newtheorem{conjecture}[section]{Conjecture}

\newcommand\R{{\mathbb R}}

\newcommand\trace{{\rm tr}}

\newcommand\SU{{\rm SU}}

\newcommand\fre{{\mathfrak e}}
\newcommand\frg{{\mathfrak g}}
\newcommand\frh{{\mathfrak h}}

\newcommand\frn{{\mathfrak n}}
\newcommand\frr{{\mathfrak r}}

\begin{document}
\title[]{Symplectic Half-Flat solvmanifolds}


\author{M. Fern\'andez}
\address[M. Fern\'andez, V. Manero]
{Universidad del Pa\'{\i}s Vasco\\
Facultad de Ciencia y Tecnolog\'{\i}a, Departamento de Matem\'aticas\\
Apartado 644, 48080 Bilbao\\
Spain}
\email{marisa.fernandez@ehu.es}
\email{victormanuel.manero@ehu.es}

\author{V. Manero}

\author{A. Otal}
\address[A. Otal, L. Ugarte]
{Departamento de Matem\'aticas\,-\,I.U.M.A.\\
Universidad de Zaragoza\\
Campus Plaza San Francisco\\
50009 Zaragoza, Spain}
\email{aotal@unizar.es}
\email{ugarte@unizar.es}

\author{L. Ugarte}


\maketitle

\begin{abstract}
We classify solvable Lie groups admitting left invariant
symplectic half-flat structure. When the Lie group has a compact quotient
by a lattice, we show that these structures
provide solutions of supersymmetric equations of
type IIA.
\end{abstract}


\section{Introduction}
An SU(3) structure on a 6-dimensional manifold defines a nondegenerate 2-form~$F$, an almost-complex structure $J$ and a complex volume form $\Psi$.
The SU(3) structure is \emph{half-flat} if the 4-form $\sigma=F \wedge F$ and the 3-form $\rho=\Psi_+$ given by the real part of the complex volume form are both closed differential forms. Half-flat SU(3) structures are of interest in both differential geometry and physics, since they give rise to $G_2$ holonomy metrics in dimension seven
by solving a certain system of evolution equations \cite{H2,VTSS}.

Nilpotent Lie algebras with half-flat structures have
been classified by Conti~\cite{C}.
Schulte-Hengesbach has classified in \cite{S} direct sums of two 3-dimensional Lie algebras admitting half-flat SU(3) structure,
and the complete classification of decomposable half-flat Lie algebras is achieved by Freibert and Schulte-Hengesbach
in \cite{FS1}. Moreover, in the recent paper \cite{FS2} they classify arbitrary indecomposable Lie algebras
admitting a half-flat SU(3) structure, except for the solvable case with 4-dimensional nilradical.
We use here some of their results as it is explained below.

In the present paper we consider the case when $F$ is itself closed, i.e. $(F,\Psi)$ is a \emph{symplectic half-flat} structure.
It is known that a symplectic half-flat structure $(F,\Psi)$
on a $6$-manifold $M$ defines, on the $7$-manifold $M\times{\mathbb R}$,
the $3$-form $\varphi=F\wedge dt+\Psi_+$
which is a calibrated $G_2$ form in the sense
of Harvey and Lawson \cite{HL}.
Moreover, as
we recall below, in the compact case this kind of structures
are closely related to solutions of the supersymmetric equations of type IIA.
Nilpotent Lie algebras having symplectic half-flat structure are classified by
Conti and Tomassini in~\cite{CT}.
In this paper, we classify the (non-nilpotent) solvable Lie algebras admitting symplectic half-flat structure and,
as an application, solutions of such equations are given.

\smallskip

For unimodular solvable Lie algebras, the classification
is the following:

\begin{theorem}\label{clasif}
An unimodular (non-Abelian) solvable Lie algebra $\frg$ has a symplectic half-flat structure if and only if it is isomorphic to one in the following list:
$$
\begin{array}{rcl}
&&
\fre(1,1)\oplus\fre(1,1) = (0,-e^{13},-e^{12},0,-e^{46},-e^{45}) ;\\[5pt]
&&
\frg_{5,1}\!\oplus\R = (0,0,0,0,e^{12},e^{13}) ;\\[5pt]
&&
A_{5,7}^{-1,-1,1}\!\oplus\R = (e^{15},-e^{25},-e^{35},e^{45},0,0) ;\\[5pt]
&&
A_{5,17}^{\alpha,-\alpha,1}\!\oplus\R= (\alpha e^{15}\!\!+\!e^{25},-e^{15}\!\!+\!\alpha e^{25},-\alpha e^{35}\!\!+\!e^{45},
-e^{35}\!\!-\!\alpha e^{45},0,0),
\, \alpha\geq 0;
\end{array}
$$
$$
\begin{array}{rcl}
&&
\frg_{6,N3} = (0,0,0,e^{12},e^{13},e^{23}) ;\\[3pt]
&&
\frg_{6,38}^{0} = (e^{23},-e^{36},e^{26},e^{26}-e^{56},e^{36}+e^{46},0) ;\\[4pt]
&&
\frg_{6,54}^{0,-1} = (e^{16}+e^{35},-e^{26}+e^{45},e^{36},-e^{46},0,0) ;\\[4pt]
&&
\frg_{6,118}^{0,-1,-1} = (-e^{16}+e^{25},-e^{15}-e^{26},e^{36}-e^{45},e^{35}+e^{46},0,0).
\end{array}
$$
\end{theorem}

It is worth noting that the corresponding solvable Lie groups admit a co-compact discrete subgroup
(see Remarks~\ref{lattice1} and~\ref{lattice2} for details).

In the description of the Lie algebras, we are using the structure equations
with respect to a basis $e^1,\ldots,e^6$ of the dual $\frg^*$. For instance,
$\fre(1,1)\oplus\fre(1,1) = (0,-e^{13},-e^{12},0,-e^{46},-e^{45})$ means that
there is a basis
$\{e^j\}_{j=1}^6$ satisfying
$d e^1=0$, $d e^2=-e^{1}\wedge e^{3}$, $d e^3=-e^{1}\wedge e^{2}$,
$d e^4=0$, $d e^5=-e^{4}\wedge e^{6}$ and $d e^6=-e^{4}\wedge e^{5}$; equivalently, the Lie bracket is
given in terms of its dual basis $\{e_j\}_{j=1}^6$ by
$[e_1,e_2]=e_3$, $[e_1,e_3]=e_2$, $[e_4,e_5]=e_6$ and $[e_4,e_6]=e_5$.
For listing the Lie algebras we use the same labels as in the lists given in
\cite{B,Sha,Tu}.

In Theorem~\ref{clasif} the Lie algebras
$\frg_{5,1}\!\oplus\R$ and $\frg_{6,N3}$ are the only (non-Abelian)
nilpotent Lie algebras having symplectic half-flat structure~\cite{CT}.
The Lie algebras
$\frg_{6,N3}$, $\frg_{6,38}^{0}$, $\frg_{6,54}^{0,-1}$ and $\frg_{6,118}^{0,-1,-1}$ in Theorem~\ref{clasif}
are indecomposable, whereas the first three Lie algebras and
the family $A_{5,17}^{\alpha,-\alpha,1}\!\oplus\R$
are decomposable.
The decomposable case $4\oplus 2$ is of special interest, because such Lie algebras (solvable or not)
cannot admit symplectic half-flat structure (see Proposition~\ref{section4+2} for details).

\smallskip

Regarding non-unimodular solvable Lie algebras, we have:

\begin{theorem}\label{clasif2}
A non-unimodular solvable Lie algebra $\mathfrak{g}$ has a symplectic half-flat structure if and only if it is isomorphic to one in the following list:
$$
\begin{array}{rcl}
&&
A_{6,13}^{-\frac{2}{3},\frac{1}{3},-1}=(-\frac{1}{3}e^{16}+e^{23}, -\frac{2}{3}e^{26}, \frac{1}{3}e^{36}, e^{46}, -e^{56}, 0);\\[4pt]
&&
A_{6,54}^{2,1}=(e^{16}+e^{35}, e^{26}+e^{45}, -e^{36}, -e^{46}, 2e^{56}, 0);\\[4pt]
&&
A_{6,70}^{\alpha,\frac{1}{2}\alpha}=(\frac{\alpha}{2}e^{16}-e^{26}+e^{35}, e^{16}+\frac{\alpha}{2}e^{26}+e^{45},
-\frac{\alpha}{2}e^{36}-e^{46}, e^{36}-\frac{\alpha}{2}e^{46}, \alpha e^{56}, 0);\\[4pt]
&&
A_{6,71}^{-\frac{3}{2}}=(\frac{3}{2}e^{16}+e^{25}, \frac{1}{2}e^{26}+e^{35}, -\frac{1}{2}e^{36}+e^{45}, -\frac{3}{2}e^{46}, e^{56}, 0);\\[4pt]
&&
N_{6,13}^{0,- 2,0, 2}=(- 2e^{16}, + 2e^{26}, e^{36}-e^{45}, e^{35}+e^{46}, 0, 0).
\end{array}
$$
Therefore, $\mathfrak{g}$ is indecomposable.
\end{theorem}

In the proof of Theorem~\ref{clasif} we use the classification given in \cite{M} of 6-dimensional unimodular solvable Lie algebras
admitting a symplectic form. That list is given using as starting point the original classification due to Mubarakzyanov \cite{Mubara}.
To prove Theorem~\ref{clasif2} we use the classification of Turkowski \cite{Tu} (or
Freibert's and Schulte-Hengesbach's refinement \cite{FS2} of
Mubarakzyanov \cite{Mubara} classification) of solvable Lie algebras with four dimensional (or five dimensional, respectively) nilradical,
and firstly we show which of those Lie algebras have a symplectic form
(see Proposition~\ref{5-nilradical} and Proposition~\ref{4-nilradical}).
Moreover, for each Lie algebra given in the theorems above
we show an explicit symplectic half-flat structure $(F,\Psi)$.
The last assertion of Theorem~\ref{clasif}, i.e. that the Lie group associated to
each Lie algebra
admits a co-compact discrete subgroup, follows from \cite{B,FLS,Gor,T,Y}.

The paper is structured as follows. In section~\ref{section-shf} we review general facts about SU(3) structures
and explain the method that we follow for studying the existence of symplectic half-flat structure on
solvable Lie algebras.
In Section~\ref{sectionindecomp} we investigate the case when the solvable Lie algebra is indecomposable
and, in particular, when the Lie algebra is non-unimodular and with
nilradical of dimension four or five. Here we must notice that according to
\cite{Mubara}, a solvable Lie algebra of
dimension 6 with nilradical
of dimension
lower than $4$ is decomposable or nilpotent.
The decomposable case is considered in Section~\ref{secc-descomp}. As a consequence
of the results proved in Section~\ref{secc-descomp} we conclude that the decomposable
solvable Lie algebras having symplectic half-flat structure are unimodular.

Finally, as an application of Theorem~\ref{clasif}, in Section~\ref{application} we
provide compact solutions to certain supersymmetric equations of type IIA.
Supersymmetric flux vacua with constant intermediate $SU(2)$
structure were introduced by Andriot~\cite{And}, and it is showed in \cite{FU} that
they are closely related to the existence of special classes of half-flat
structures on the internal manifold. In particular, solutions of the SUSY equations of type IIA
possess a symplectic half-flat structure.
In Proposition~\ref{IIA-solutions}
we prove that
any $6$-dimensional compact solvmanifold admitting an invariant symplectic half-flat structure
also admits a solution of the SUSY equations of type IIA. This provides the complete list of compact solvmanifolds
admitting invariant solutions of such supersymmetric equations.

\section{Symplectic half-flat structures} \label{section-shf}

In this section we recall some well-known facts about $\SU(3)$ structures
and consider several obstructions to the existence of symplectic half-flat structures.
We will follow ideas given in \cite{C}, \cite{CF} and \cite{FS1}.

An $\SU(3)$ structure on a 6-dimensional manifold $M$ is an $\SU(3)$ reduction of the frame bundle of $M$. We consider the characterization of $\SU(3)$ structures given in \cite{H1}, i.e. in terms of certain stable forms which satisfy some additional compatibility conditions.

A $3$-form $\rho$ on a six dimensional oriented vector space $(V,\nu)$ is \emph{stable} if its orbit under the action of the group GL($V$) is open.
Let $\kappa:\Lambda^5V^*\longrightarrow V\otimes\Lambda^6V^*$ be the isomorphism with $\kappa(\eta)=X\otimes\nu$ such that $\iota_X\nu=\eta$, and let $K_\rho:V\longrightarrow V\otimes\Lambda^6V^*$ be given by $K_\rho(X)=\kappa(\iota_X\rho\wedge\rho)$.
In terms of the invariant $\lambda(\rho)=\frac{1}{6}\trace(K_\rho^2)$, the stability of $\rho$ is
equivalent to the open condition $\lambda(\rho)\neq 0$. Moreover, $\lambda(\rho)$ enables us to construct a volume form $\phi(\rho):=\sqrt{|\lambda(\rho)|}\in\Lambda^6V^*$.

The endomorphism $J_\rho:=\frac{1}{\phi(\rho)}K_\rho$ gives rise to an almost complex structure if $\lambda(\rho)<0$.
The action of $J_\rho^*$ on 1-forms is given by the formula
\begin{equation}\label{def-J}
J_\rho^*\alpha(X)\phi(\rho)=\alpha\wedge\iota_X\rho\wedge\rho.
\end{equation}
The characterization of $\SU(3)$ structures requires the existence of a $2$-form $F$ which is
\emph{stable}, i.e. $\phi(F):=F^3\neq 0$, such that the pair $(F,\rho)$ is \emph{compatible},
in the sense that
$$
F\wedge\rho=0,
$$
and \emph{normalized}, i.e.
$$
\phi(\rho)=2\phi(F).
$$
Such a pair $(F,\rho)$ induces a pseudo Euclidean metric $g(\cdot,\cdot)=F(J_\rho\cdot,\cdot)$
which satisfies on 1-forms the identity
\begin{equation}\label{eq}
\alpha\wedge J^*_\rho\beta\wedge F^2=\frac{1}{2}g(\alpha,\beta)F^3,\quad\alpha,\beta\in V^*.
\end{equation}

An $\SU(3)$ structure on $V$ is a pair of compatible and normalized stable forms $(F,\rho)\in\Lambda^2V^*\times\Lambda^3V^*$ with $\lambda(\rho)<0$ inducing a positive-definite metric. If in addition $V=\frg$ is a Lie algebra then a \emph{symplectic half-flat structure} on $\frg$ is an $\SU(3)$ structure $(F,\rho)$ such that $dF=0$ and $d\rho=0$, where $d$ denotes the Chevalley-Eilenberg differential of $\frg$.

From now on, we denote by $Z^k(\frg)$ the space of closed $k$-forms on $\frg$, by $\mathcal{S}(\frg)=\{F\in Z^2(\frg)\mid\, F^3\not=0\}$
the space of symplectic forms on $\frg$, and by ${\rm Ann}(\rho)$ the annihilator of $\rho$ in the
exterior algebra $\Lambda^*\frg^*$.
In \cite{FS1} the authors give a useful simple obstruction to the existence of half-flat structures on Lie algebras based on equation (\ref{eq}).
In the presence of a compatible symplectic form, such obstruction reads as:

\begin{proposition}\label{criterion}
Let us fix a volume element $\nu$ on $\frg$ and let $\rho\in Z^3(\frg)$.
Let us define $\tilde{J}_\rho^*$ by imitating \eqref{def-J}
but with respect to the volume element $\nu$, that is,
$$
(\tilde{J}_\rho^*\alpha)(X)\,\nu=\alpha\wedge\iota_X\rho\wedge\rho, \quad X\in\frg,
$$
and let $\tilde{J}_\rho\colon \frg\longrightarrow \frg$ be the
endomorphism of $\frg$
given by $\alpha(\tilde{J}_\rho X)=-(\tilde{J}_\rho^*\alpha)(X)$, for any $X\in\frg$
and $\alpha\in\frg^*$. Then:
\begin{enumerate}
\item[{\rm (i)}] If there is some $\alpha\in \frg^*$ such that
$$
\alpha\wedge \tilde{J}^*_\rho\alpha\wedge F^2=0
$$
for any $F\in \mathcal{S}(\frg)\cap {\rm Ann}(\rho)$, then $\frg$ does not admit any symplectic half-flat structure.
\item[{\rm (ii)}] If there are some $X,Y\in\frg$ such that
$$
F(\tilde{J}_\rho(X),X)\,F(\tilde{J}_\rho(Y),Y)\leq 0
$$
for any $F\in \mathcal{S}(\frg)\cap {\rm Ann}(\rho)$, then $\frg$ does not admit any symplectic half-flat structure.
\end{enumerate}
\end{proposition}

\begin{proof}
Notice that $\tilde{J}_\rho$ is proportional to the almost complex structure $J_\rho$. Now, from (\ref{eq}) it follows that
in case (i) the induced metrics $g$ are degenerate, and in case (ii) we get a contradiction with the positive-definiteness
of $g$.
\end{proof}

In \cite{CF} it is proved the following restriction to the existence of a calibrated
$G_2$ form on a Lie algebra $\mathfrak{h}$: if there is a a nonzero vector
$X\in \mathfrak{h}$ such that $(i_X\phi)^3=0$ for all $\phi\in Z^3(\mathfrak{h}^*)$,
then $\mathfrak{h}$ does not admit calibrated $G_2$ structures.
It is known \cite{FernandezGray} that a symplectic half-flat structure on the
Lie algebra $\mathfrak{g}$ induces the calibrated $G_2$ form
$\varphi=F\wedge dt+\Psi_+$ on the Lie algebra
$\mathfrak{h}=\mathfrak{g}\oplus\mathbb{R}$.
These facts imply the following:
\begin{proposition}\label{criterion2}
If there exists $X\in \mathfrak{g}\oplus\mathbb{R}$ such that
$$(i_X\varphi)^3=0 \text{ for all } \varphi\in Z^3((\mathfrak{g}\oplus\mathbb{R})^*),$$
then $\mathfrak{g}$ does not admit symplectic half-flat structure.
\end{proposition}

In this paper we study the existence of symplectic half-flat structures on 6-dimensional solvable Lie algebras.
In the decomposable cases $3\oplus3$, $4\oplus2$ and $5\oplus1$ our starting point is the half-flat classification
obtained in \cite{S} and \cite{FS1}
and, after finding which of them admit symplectic forms, we apply the above proposition to classify the Lie algebras having symplectic half-flat
structure.
We remark that the results obtained in the decomposable cases are completely general and not constrained to the solvable type.

On the other hand, in the case of indecomposable
unimodular solvable Lie algebras,
we start with the symplectic classification given in \cite{M} and then we search
for a compatible half-flat structure (see Proposition~\ref{unimod-indescomp-list}).
For indecomposable non-unimodular
solvable Lie algebras with nilradical of dimension $4$, we use the
list of \cite{Tu} and we show which of them have a symplectic form (see
Proposition \ref{4-nilradical}).
However, in the case of indecomposable non-unimodular
solvable Lie algebras, with $5$-dimensional nilradical, we begin with the classification given in \cite{FS2}
of indecomposable
solvable Lie algebras having a half-flat structure (see
Proposition \ref{5-nilradical}).
The next step is then to apply
the above
propositions which are obstructions to the existence
of symplectic half-flat
structures.

In the cases of existence of symplectic half-flat structure we will only provide a particular example $(F,\rho)$ written in standard form
$F=f^{12}+f^{34}+f^{56}$ and $\rho=Re((f^{1}+i\,f^{2})(f^{3}+i\,f^{4})(f^{5}+i\,f^{6}))$ with respect to some basis $\{f^1,\ldots,f^6\}$ of~$\frg^*$.
Notice that this basis is orthonormal for the underlying metric and the almost complex structure $J$ is given by
$J^* f^1=-f^2$, $J^* f^3=-f^4$ and $J^* f^5=-f^6$.

\section{The indecomposable case} \label{sectionindecomp}
In this section
we prove Theorem~\ref{clasif} and Theorem~\ref{clasif2} in the case
of indecomposable solvable Lie algebras of dimension six.

In the next proposition we focus on unimodular solvable Lie algebras which are not nilpotent,
because it is proved in~\cite{CT} that $\frg_{6,N3}$ is
the only indecomposable nilpotent Lie algebra admitting symplectic half-flat structure.
For listing the unimodular solvable Lie algebras, we use the notation given in \cite{B}
(see Appendix).

\begin{proposition}\label{unimod-indescomp-list}
The only 6-dimensional indecomposable unimodular non-nilpotent solvable Lie algebras admitting symplectic half-flat structures are
$\frg_{6,38}^{0}$, $\frg_{6,54}^{0,-1}$ and $\frg_{6,118}^{0,-1,-1}$.
\end{proposition}

\begin{proof}
By \cite[Theorem 2]{M} the indecomposable unimodular non-nilpotent solvable Lie algebras admitting a symplectic form are
those appearing in Table~1 of the Appendix.

Notice that the Lie algebra $\frg_{6,54}^{0,-1}$ has a symplectic half-flat structure by \cite{TV}.
Moreover, explicit symplectic half-flat structures on $\frg_{6,38}^{0}$ and $\frg_{6,118}^{0,-1,-1}$
are given in Table~1.

Next we show in some detail how Proposition~\ref{criterion} is used for the remaining Lie algebras $\frg$ in the list.
In all cases we consider on $\frg$ the volume element given by $\nu=e^{123456}$,
where $\{e^{1},\ldots,e^6\}$ is the basis of $\frg^*$ in Table~1.

On the Lie algebra $\frg=\frg_{6,3}^{0,-1}$, any pair $(F,\rho)$ with $F\in Z^2(\frg)$ and $\rho\in Z^3(\frg)$
is given by
$$
\begin{array}{l}
F=b_1e^{16}+b_2e^{23}+b_3e^{26}+b_4e^{36}+b_5e^{45}+b_6e^{46}+b_7e^{56},\\[5pt]
\rho=a_1e^{123}+a_2e^{126}+a_3e^{136}+a_4e^{146}+a_5e^{156}+a_6e^{236}+a_7e^{246}\\
\phantom{m}+a_8e^{256}+a_9e^{345}+a_{10}e^{346}+a_{11}e^{356}+a_{12}e^{456}.
\end{array}
$$
The stability of $F$ implies that $b_1,b_2,b_5\neq 0$, so we can take without loss of generality $b_5=1$.
From the compatibility condition $F\wedge\rho=0$ we get $a_1=a_2=a_4=a_5=0$, $a_3=a_9b_1$ and $a_6= a_9b_3-a_{12}b_2$.
Now,
Proposition~\ref{criterion}~(i)
is satisfied for the 1-form $\alpha=e^6$ and consequently $\frg$ does not admit symplectic half-flat structure.

For the Lie algebra $\frg=\frg_{6,13}^{-1,\frac{1}{2},0}$, any pair $(F,\rho)\in Z^2(\frg)\times Z^3(\frg)$
is given by
$$
\begin{array}{l}
F=b_1 e^{13}+b_2(-\frac{1}{2}e^{16}+e^{23})+b_3e^{24}+b_4e^{26}+b_5e^{36}+b_6e^{46}+b_7e^{56},\\[5pt]
\rho=a_1e^{126}+a_2e^{135}+a_3e^{136}+a_4(\frac{1}{2}e^{146}-e^{234})+a_5(\frac{1}{2}e^{156}+e^{235})+a_6e^{236}\\
\phantom{m} +a_7e^{245}+a_8e^{246}+a_9e^{256}+a_{10}e^{346}+a_{11}e^{356}+a_{12}e^{456}.
\end{array}
$$
The form $F$ is stable if and only if $b_1, b_3, b_7\neq 0$, so we can suppose without loss of generality
that $b_1=1$. When imposing $F\wedge\rho=0$ we find that $a_7=-a_2b_3$, $a_8 = -a_3b_3$,
$a_9 = a_5 b_2 + a_2 b_4$, $a_5 = a_2b_2$, $a_{12} = a_2b_6$ and $a_4 = \frac{-a_{11} b_3 - a_2 b_3 b_5}{b_7}$.
A direct calculation shows that
$$
\tilde{J}_\rho e_1=-\frac{b_3 (a_{11} b_2 + b_2 b_5 + (a_2-1) a_3 b_7)}{b_7}e_1+\frac{a_2 b_3 (a_{11} + b_5)}{b_7}e_2+2a_1a_2\,e_4,
$$
which implies that $F(\tilde{J}_\rho e_1,e_1)=0$ and so Proposition~\ref{criterion}~(ii) is
satisfied for $X=Y=e_1$.

On the Lie algebra $\frg=\frg_{6,70}^{0,0}$,
any pair $(F,\rho)\in Z^2(\frg)\times Z^3(\frg)$ is given by
$$
\begin{array}{l}
F=b_1(e^{13}+e^{24})+b_2(e^{16}+e^{45})+b_3(e^{26}-e^{35})+b_4e^{34}+b_5e^{36}+b_6e^{46}+b_7e^{56},\\[5pt]
\rho=a_1e^{125}+a_2(e^{135}+e^{245})+a_3e^{136}+a_4(e^{145}-e^{235})+a_5(e^{146}+e^{236})\\
\phantom{m}+a_6e^{156}+
a_7e^{246}+a_8e^{256}+a_9e^{345}+a_{10}e^{346}+a_{11}e^{356}+a_{12}e^{456}.
\end{array}
$$
The form $F$ is stable if and only if $b_1, b_7\neq 0$, so we consider $b_1=1$. From the condition $F\wedge\rho=0$ we get that $a_2=\frac{a_1b_4}{2}$, $a_3=-a_7$,
$a_8=a_2b_3-a_1b_5+a_4b_2$, $a_6 = a_2 b_2 + a_1 b_6$, $a_{11} = (a_4 + a_5) b_2 + (a_7 - a_9) b_3 + a_8 b_4 + a_2 b_5$ and
$a_{12} = (a_9-a_3) b_2 - (a_4 + a_5) b_3 - a_6 b_4 + a_2 b_6$.
We find that $\tilde{J}_\rho e_1=\alpha_1e_1+\dots+\alpha_4e_4$ and $\tilde{J}_\rho e_2=\beta_1e_1+\dots+\beta_4e_4$
with $\beta_4=-\alpha_3=-2a_1 (a_4 + a_5)$.
This implies that
$$
F(\tilde{J}_\rho e_1,e_1)=F(\alpha_1e_1+\dots+\alpha_4e_4,e_1)=-\alpha_3=-2a_1 (a_4 + a_5)$$
and
$$
F(\tilde{J}_\rho e_2,e_2)=F(\beta_1e_1+\dots+\beta_4e_4,e_2)=-\beta_4=2a_1 (a_4 + a_5).$$
Therefore, Proposition~\ref{criterion}~(ii) is satisfied for $X=e_1$ and $Y=e_2$, and $\frg_{6,70}^{0,0}$
does not admit symplectic half-flat structure.

For the Lie algebras $\frg_{6,10}^{0,0}$, $\frg_{6,18}^{-1,-1}$, $\frg_{6,21}^0$ and
$\frg_{6,36}^{0,0}$ one can prove that Proposition~\ref{criterion}~(i)
is satisfied for the 1-form $\alpha=e^6$ and consequently they do not admit symplectic half-flat structure.

For the Lie algebras $\frg_{6,13}^{\frac{1}{2},-1,0}$ and $\frg_{6,78}$ a similar argument proves that $F(\tilde{J}_\rho e_1,e_1)=0$,
for $\frg_{6,15}^{-1}$ one has that $F(\tilde{J}_\rho e_4,e_4)=0$ and for $\frn_{6,84}^{\pm 1}$ we have $F(\tilde{J}_\rho e_2,e_2)=0$.
Thus, by
Proposition~\ref{criterion}~(ii) these Lie algebras do not admit symplectic half-flat structure.
\end{proof}

\begin{remark}\label{lattice1}
{\rm
The solvable Lie group corresponding to $\frg_{6,38}^{0}$ admits a lattice by \cite[Proposition 8.3.3]{B}.
For $\frg_{6,54}^{0,-1}$, it is shown in \cite{FLS} that the corresponding simply-connected Lie group admits a compact quotient.
Finally, the solvable Lie group corresponding to
$\frg_{6,118}^{0,-1,-1}$ admits a lattice by \cite{Y}.
}
\end{remark}

In the following we study non-unimodular Lie algebras. A
solvable Lie algebra of dimension 6 with nilradical
of dimension lower than $4$ is decomposable or nilpotent \cite{Mubara}.
So, we are left to study Lie algebras with nilradical of dimension $4$ and $5$.

To prove next proposition we use the list of Freibert and Schulte-Hengesbach
of the solvable Lie algebras, with 5-dimensional nilradical,
having a half-flat structure~\cite{FS2}.
There the authors use
the corrected version, due to Shabanskaya \cite{Sha},
of the original classification by Mubarakzyanov \cite{Mubara} of solvable Lie algebras of dimension $6$ with $5$-dimensional nilradical.
That list given in \cite{Sha} contains $52$ Lie algebras and $70$ families depending at least
of one-parameter: $44$ one-parameter families, $22$ two-parameter families, $3$ three-parameter families and $1$ four-parameter family.

\begin{proposition}\label{5-nilradical}
The only 6-dimensional non-unimodular solvable Lie algebras with $5$-dimensional nilradical admitting symplectic half-flat structures are $A_{6,13}^{-\frac{2}{3},\frac{1}{3},-1}$, $A_{6,54}^{2,1}$, $A_{6,70}^{\alpha,\frac{\alpha}{2}}\, (\alpha \neq 0)$ and $A_{6,71}^{-\frac{3}{2}}$.
\end{proposition}

\begin{proof}
From the list given in \cite{FS2} one can check that the unique non-unimodular solvable Lie algebras
with $5$-dimensional nilradical admitting a symplectic form are those given in Table~2 of the Appendix.
Next we use Proposition~\ref{criterion2} to find which of them admit in addition a symplectic half-flat structure.

Let $\mathfrak{g}=A_{6,39}^{\frac{3}{2},-\frac{3}{2}}$ and consider the 7-dimensional Lie algebra $\mathfrak{h} =\mathfrak{g}\oplus\mathbb{R}$. Let us consider an arbitrary $\varphi \in Z^3(\mathfrak{h^*})$.
Then,
$$
\begin{array}{rl}
\varphi\!\!\!&=a_1e^{146}+a_2e^{156}+a_3(e^{135}+2e^{236})+a_4e^{245}-a_5(e^{145}-e^{246})+a_6e^{256}\\[5pt]
&\phantom{i}+a_7(2e^{157}+e^{267})-a_8(e^{136}-e^{345})+a_9
e^{346}+a_{10}e^{347}+a_{11}e^{356}+a_{12}e^{367}\\[5pt]
&\phantom{i}+e_{13}e^{456}-a_{14}(e^{167}-2e^{457})+a_{15}e^{467}+a_{16}e^{567}.
\end{array}
$$
Put $\nu=i_{e_1}\varphi$. Thus,
$$\nu=a_1e^{46}+a_2e^{56}+a_3e^{35}-a_5e^{45}+2a_7e^{57}-a_8e^{36}-a_{14}e^{67}$$
is a degenerate 2-form and Proposition \ref{criterion2} is fulfilled with the vector $e_1$, so
$\mathfrak{g}$ does not admit symplectic half-flat structures.

Similarly, for $\mathfrak{g}=A_{6,39}^{1,-1}$ one has that Proposition~\ref{criterion2} is satisfied for the vector~$e_2$.

Now, if $\mathfrak{g}$ is the Lie algebra $A_{6,42}^{-1}$ then Proposition~\ref{criterion2} is satisfied for $e_2$.

For the Lie algebras $\mathfrak{g}=A_{6,51}^{\pm 1}$,
in both cases Proposition~\ref{criterion2}
can be checked with the vector $e_3$.

Now, if $\mathfrak{g}$ is the Lie algebra $A_{6,54}^{-1,-2}$ then Proposition~\ref{criterion2} is satisfied for $e_1$.

We study now the family of Lie algebras $A_{6,54}^{\alpha,\alpha-1}$ with $0 < \alpha \leq 2$. If $\alpha=0$ then it is unimodular, and the case $\alpha=2$ will be studied at the end of the proof. On this family Proposition~\ref{criterion2} is satisfied for the vector $e_1$.

For $A_{6,56}^{1}$ one gets that Proposition~\ref{criterion2} is satisfied for the vector $e_1$, hence $A_{6,56}^{1}$ does not admit symplectic half-flat structure.

The Lie algebra $A_{6,65}^{1,2}$
does not admit a symplectic half-flat structure since Proposition~\ref{criterion2} is satisfied for the vector $e_2$.

For the Lie algebra $A_{6,76}^{-3}$ we have Proposition~\ref{criterion2} fulfilled for the vector $e_1$.

On the Lie algebra $A_{6,82}^{2,5,9}$, vector $e_1$ satisfies Proposition~\ref{criterion2}.

The Lie algebra $A_{6,94}^{-3}$ has no symplectic half-flat structure because Proposition~\ref{criterion2} is satisfied for
the vector $e_1$.

On the Lie algebra $A_{6,94}^{-\frac{5}{3}}$, vector $e_1$ satisfies Proposition~\ref{criterion2}.

The Lie algebra $A_{6,94}^{-1}$ has no symplectic half-flat structure because Proposition~\ref{criterion2} is satisfied for
the vector $e_1$.

Finally, explicit symplectic half-flat structures on $A_{6,13}^{-\frac{2}{3},\frac{1}{3},-1}$,
$A_{6,54}^{2,1}$, $A_{6,70}^{\alpha,\frac{\alpha}{2}}\, (\alpha\neq 0)$ and $A_{6,71}^{-\frac{3}{2}}$
are given in Table~2 of the Appendix.
\end{proof}

To complete the proof of Theorem \ref{clasif2} in the
indecomposable case, it remains to study the solvable Lie
algebras with 4-dimensional nilradical. For this, we use the list of \cite{Tu} that contains
$12$ Lie algebras and $31$ families depending at least
of one-parameter. Indeed,  there are $14$ one-parameter families, $10$ two-parameter
families, $4$ three-parameter families and  $3$ four-parameter.

\begin{proposition}\label{4-nilradical}
$N_{6,13}^{0,-2,0,2}$ is the only 6-dimensional non-unimodular solvable Lie algebra with 4-dimensional nilradical admitting a symplectic half-flat structure.
\end{proposition}

\begin{proof}
Starting from the list of Turkowski \cite{Tu} of non-unimodular solvable Lie algebras with 4-dimensional nilradical,
we reduce our attention to those admitting a symplectic form, which we list in Table~3 of
the Appendix.

An explicit symplectic half-flat structure $(F,\rho)$ on the Lie algebra $N_{6,13}^{0,-2,0,2}$ is given in Table~3.
For any of the remaining Lie algebras we use
Propositions~\ref{criterion} and~\ref{criterion2} to prove non-existence of symplectic half-flat structures.
Indeed, for the families $N_{6,2}^{-1,\beta,-\beta}, N_{6,7}^{0,\beta,0}\, (\beta\not=0)$ and $N_{6,13}^{\alpha,\beta,-\alpha,-\beta}$, with $\alpha^2+\beta^2\neq 0$ and $\beta\neq \pm2$, as well as for the Lie algebra $N_{6,17}^{0}$
the hypothesis of Proposition \ref{criterion2} is satisfied with $X=e_3$. The family of Lie algebras $N_{6,2}^{0,-1,\gamma}$ satisfies Proposition \ref{criterion2} with the vector $X=e_2$ for all $\gamma$. Almost all of the remaining Lie algebras and parameter families appearing in Table~3 satisfy the hypothesis of Proposition \ref{criterion2} with the vector $X=e_1$.
However, for the Lie algebra $N_{6,28}$ and the families $N_{6,1}^{\alpha,\beta,-\alpha,-\beta}, N_{6,1}^{\alpha,\beta,0,-1}$ and $N_{6,1}^{\alpha,\beta,-1,0}$, where Proposition \ref{criterion2} is not enough, we use
Proposition~\ref{criterion}
to assert that they do not admit symplectic half-flat structure.
For this, we proceed as in the proof of Proposition~\ref{unimod-indescomp-list} for $A_{6,70}^{0,0}$.
We see that the Lie algebra $N_{6,28}$ satisfies
Proposition~\ref{criterion}~(ii) for $X=e_5$ and $Y=e_6$.
For the family $N_{6,1}^{\alpha,\beta,-\alpha,-\beta}$ the obstruction (i) in Proposition~\ref{criterion} is satisfied for
$(2+\sqrt{3})e^5+e^6$.
For $N_{6,1}^{\alpha,\beta,0,-1}$ Proposition~\ref{criterion}~(i) is satisfied for
the 1-form $(1+\sqrt{3})e^5+e^6$ or $(1+\sqrt{3})e^5-e^6$ depending on the sign of $\beta$.
Finally,
for the family $N_{6,1}^{\alpha,\beta,-1,0}$ the obstruction (i) in Proposition~\ref{criterion} is satisfied for the
1-form $(1+\sqrt{3})e^5+2e^6$ or $(1+\sqrt{3})e^5-2e^6$ depending on the sign of $\alpha$,
which completes the proof.
\end{proof}

\section{The decomposable case} \label{secc-descomp}
In this section we consider all 6-dimensional Lie algebras (unimodular and
non-unimodular) of the form $\frg=\frg_1\oplus\frg_2$.
Our starting point is the half-flat classification given in \cite{S} for the 3$\oplus$3 case and in \cite{FS1} for the 4$\oplus$2 and 5$\oplus$1 cases.

As in the previous section, we will suppose that $\frg$ is non-nilpotent, because it is proved in~\cite{CT} that
the only decomposable nilpotent Lie algebras having symplectic half-flat structure are $\frg_{5,1}\!\oplus\R$ and the Abelian Lie algebra.
Therefore, Theorem~\ref{clasif} follows as a consequence of \cite{CT}, Proposition~\ref{unimod-indescomp-list} and
Propositions~\ref{section3+3}, \ref{section4+2} and~\ref{section5+1} below.

In the next result
we consider the 6-dimensional Lie algebras of the form $\frg=\frg_1\oplus\frg_2$ with $\dim\frg_1=\dim \frg_2=3$.

\begin{proposition}\label{section3+3}
The only $3\oplus3$ non-nilpotent Lie algebra which admits a symplectic half-flat structure is $\fre(1,1)\oplus\fre(1,1)$.
\end{proposition}

\begin{proof}
Starting from the classification in \cite{S}, one can check that the $3\oplus3$ non-nilpotent Lie algebras
admitting both half-flat structure and symplectic form are those listed in Table~4 of the Appendix.
(Notice that in particular the only $3\oplus3$ non-unimodular solvable Lie algebras
having both symplectic and half-flat structure are $\fre(2)\oplus\frr_2\oplus\R$ and $\fre(1,1)\oplus\frr_2\oplus\R$.)

A direct calculation shows that the Lie algebra $\fre(2)\oplus\fre(2)$ satisfies
Proposition~\ref{criterion}~(ii) for $X=e_1$ and $Y=e_2$.

For the Lie algebra $\fre(2)\oplus\R^3$ we have that $F(\tilde{J}_\rho e_2,e_2)=0$.
Let $\frg=\fre(1,1)\oplus\R^3$, since $Z^2(\frg)=Z^2(\fre(2)\oplus\R^3)$ and $Z^3(\frg)=Z^3(\fre(2)\oplus\R^3)$ the
same argument as in the previous case is valid, so $F(\tilde{J}_\rho e_2,e_2)=0$.

The Lie algebra $\fre(2)\oplus\fre(1,1)$ satisfies Proposition~\ref{criterion}~(ii) for $X=e_2$ and $Y=e_3$.

For $\fre(2)\oplus\frh$ and $\fre(1,1)\oplus\frh$ one has that $F(\tilde{J}_\rho e_6,e_6)=0$.

The Lie algebras $\fre(2)\oplus\frr_2\oplus\R$ and $\fre(1,1)\oplus\frr_2\oplus\R$ both satisfy the obstruction~(i)
in Proposition~\ref{criterion} for $\alpha=e^4$.

Finally, it is proved in \cite{TV} that the Lie algebra $\fre(1,1)\oplus\fre(1,1)$ has a symplectic half-flat structure
(see Table~4).
\end{proof}

In the next result we consider
a six dimensional Lie algebra $\frg=\frg_1\oplus\frg_2$ with $\rm{dim}\ \frg_1=4$ and $\rm{dim}\ \frg_2=2$.

\begin{proposition}\label{section4+2}
There is no $4\oplus2$ non-nilpotent Lie algebra admitting symplectic half-flat structure.
\end{proposition}

\begin{proof}
All the non-nilpotent Lie algebras of type $4\oplus 2$ with decomposable four-dimensional one are isomorphic to one of
Proposition~\ref{section3+3} except the Lie algebra $\mathfrak{r_2}\oplus\mathfrak{r_2}\oplus\mathfrak{r_2}$. So in order to avoid the study of the same Lie algebras several times, from now on we will refer to the $4\oplus2$ Lie algebras as the ones that being decomposable in the form $4\oplus2$ cannot be decomposed as $3\oplus3$.
The only $4\oplus2$ non-nilpotent Lie algebras admitting both half-flat structure \cite{FS1} and symplectic form are
the non-unimodular Lie algebras
given in Table~5.
These four possible cases are easily rejected thanks to the obstruction (i) in Proposition~\ref{criterion}:
it suffices to consider the 1-form $\alpha=e^5$ for $A_{4,1}\oplus\frr_2$, $\alpha=e^4$ for $A_{4,9}^{-\frac{1}{2}}\oplus\frr_2$,
$\alpha=e^3$ in case $A_{4,12}\oplus\frr_2$ and $\alpha=e^1+e^3$ for $\frr_2\oplus\frr_2\oplus\frr_2$.
\end{proof}

In the next result
we study the $5\oplus1$ decomposable Lie algebras, i.e. $\frg=\frg_1\oplus\R$.

\begin{proposition}\label{section5+1}
The only $5\oplus1$ non-nilpotent Lie algebras admitting symplectic half-flat structures are
$A_{5,7}^{-1,-1,1}\oplus\R$ and $A_{5,17}^{\alpha,-\alpha,1}\oplus\R$ with $\alpha\geq 0$.
\end{proposition}

\begin{proof}
As in the previous proposition, every Lie algebra of type $5\oplus1$ with decomposable five-dimensional part is isomorphic to one of types $3 \oplus 3$ or $4 \oplus 2 $. In the list of $5\oplus1$ non-nilpotent Lie algebras admitting half-flat structures \cite{FS1}, those having also symplectic form are
given in Table~6 of the Appendix.
(In particular, the only $5\oplus1$ non-unimodular solvable Lie algebras
having both half-flat structure and symplectic form are $A_{5,36}\oplus\R$ and $A_{5,37}\oplus\R$.)

In the case of $A_{5,7}^{-1,\beta,-\beta}\oplus\R$ with $0<\beta<1$, $A_{5,8}^{-1}\oplus\R$, $A_{5,19}^{-1,2}\oplus\R$, $A_{5,36}\oplus\R$ and $A_{5,37}\oplus\R$ we find that
obstruction (ii) in Proposition~\ref{criterion} is satisfied
because $F(\tilde{J}_\rho e_1,e_1)=0$.

As $Z^2(A_{5,13}^{-1,0,\gamma}\oplus\R)=Z^2(A_{5,17}^{0,0,\gamma}\oplus\R)=Z^2(A_{5,7}^{-1,\beta,-\beta}\oplus\R)$ and $Z^3(A_{5,13}^{-1,0,\gamma}\oplus\R)=Z^3(A_{5,17}^{0,0,\gamma}\oplus\R)=Z^3(A_{5,7}^{-1,\beta,-\beta}\oplus\R)$
for any $\gamma$,
the same conclusion as for $A_{5,7}^{-1,\beta,-\beta}\oplus\R$ is valid.
Moreover, we have that $Z^2(A_{5,14}^0\oplus\R)=Z^2(A_{5,8}^{-1}\oplus\R)$ and $Z^3(A_{5,14}^0\oplus\R)=Z^3(A_{5,8}^{-1}\oplus\R)$.
Therefore, the Lie algebras $A_{5,13}^{-1,0,\gamma}\oplus\R$, $A_{5,17}^{0,0,\gamma}\oplus\R$ and $A_{5,14}^0\oplus\R$
do not admit symplectic half-flat structure.

For the case $\frg=A_{5,15}^{-1}\oplus\R$
we have that an arbitrary pair
$(F,\rho)\in Z^2(\frg)\times Z^3(\frg)$ is given by
$$
\begin{array}{l}
F=b_1(e^{14}-e^{23})+b_2e^{15}+b_3e^{24}+b_4e^{25}+b_5e^{35}+b_6e^{45}+b_7e^{56},\\[5pt]
\rho=a_1e^{125}+a_2e^{135}+a_3e^{145}+a_4(e^{146}-e^{236})+a_5e^{156}+a_6e^{235}+a_7e^{245}\\
\phantom{m}+a_8e^{246}+a_9e^{256}+a_{10}e^{345}+a_{11}e^{356}+a_{12}e^{456}.
\end{array}
$$
It follows that $F^3\not=0$ if and only if $b_1,b_7\neq 0$, so we can take $b_1=1$.
When imposing $F\wedge\rho=0$ we have that $a_3 = a_6 - a_2 b_3$, $a_4 = a_5 = 0$, $a_9 = a_8 b_2$, $a_{11} = 0$
and $a_{12} = -a_8 b_5$.
A direct calculation shows that $\tilde{J}_\rho e_1=a_2 a_8\, e_1$,
so $F(\tilde{J}_\rho e_1,e_1)=0$.
Similarly, for $A_{5,18}^0\oplus\R$ we get that $\tilde{J}_\rho e_1$ is a multiple of $e_1$.
Thus, by Proposition~\ref{criterion} the Lie algebras $A_{5,15}^{-1}\oplus\R$ and $A_{5,18}^0\oplus\R$ do not admit
symplectic half-flat structure.

Finally, the half-flat structures given in \cite{FS1} for $A_{5,7}^{-1,-1,1}\oplus\R$ and $A_{5,17}^{\alpha,-\alpha,1}\oplus\R$
with $\alpha\geq 0$ satisfy that $dF=0$, so they provide examples of symplectic half-flat structures (see Table~6).
\end{proof}

Notice that the previous propositions imply that if $\frg$ is a decomposable non-unimodular
solvable Lie algebra then $\frg$ does not admit symplectic half-flat structure.

\begin{remark}\label{lattice2}
{\rm
The Lie algebra $\fre(1,1)\oplus\fre(1,1)$ is unimodular
and the corresponding simply connected Lie group admits a compact quotient
as it is shown in \cite{T}.
For the simply connected Lie groups corresponding to $A_{5,7}^{-1,-1,1}$
and $A_{5,17}^{\alpha,-\alpha,1}$ with $\alpha\geq 0$, conditions for the
existence of lattice are given in \cite[Propositions 7.2.1 and 7.2.14]{B}. In particular,
there is a lattice for the cases $A_{5,7}^{-1,-1,1}$ and $A_{5,17}^{\alpha,-\alpha,1}$ for
$\alpha=0$ and for some positive $\alpha$.
}
\end{remark}

\section{Symplectic half-flat structures and SUSY equations of type IIA} \label{application}

In this section we consider as in \cite{And,GMPT} type II supergravity backgrounds which are warped products of the
Minkowski space $\mathbb{R}^{3,1}$ and a compact manifold $M$ of dimension~6, with non-trivial fluxes living
on the internal manifold $M$.
In order to get (at least) $N=1$ supersymmetry, it is required the existence of (at least) a pair
of globally defined and non-vanishing spinors on $M$ satisfying the SUSY conditions.
The existence of such a pair implies a reduction of the structure group of the tangent bundle
to a subgroup $G\subset SO(6)$, so that different types of $G$ structures arise.

The existence of a globally defined non-vanishing spinor $\eta_{+}$ on $M$
defines a reduction to SU(3), that is, $M$ is endowed with an SU(3) structure $(F,\Psi)$.
On the other hand, an SU(2) structure on a $6$-dimensional manifold $M$
is defined by two orthogonal globally defined spinors $\eta_{+}$ and $\chi_{+}$,
which we can suppose of unit norm, or equivalently by an almost Hermitian structure $(J, g)$,
a $(1,0)$-form $\alpha$ such that $\| \alpha \|^2 =2$,
a real $2$-form $\omega$ and a $(2,0)$-form $\Omega$
satisfying the conditions
$$
\omega^2 =  \frac{1}{2} \Omega \wedge \overline \Omega \neq 0,\quad
\omega \wedge \Omega =0, \quad  \Omega \wedge \Omega =0
$$
and
$$
i_{\alpha} \Omega =0, \quad i_{\alpha} \omega =0,
$$
where $i_{\alpha}$ denotes the contraction by the vector field dual to $\alpha$.

Notice that the SU(2) structure $(\alpha,\omega,\Omega)$ is naturally embedded in the SU(3) structure
defined by $\eta_+$ as
\begin{equation}\label{asocSU(3)}
F = \omega + \frac{i}{2} \alpha \wedge \overline \alpha, \quad \
\Psi=\alpha \wedge \Omega;
\end{equation}
and conversely, given an SU(3) structure $(F,\Psi)$ and
a $(1,0)$-form $\alpha$ of norm~$\sqrt{2}$ on $M$, then
$\omega = F - \frac{i}{2} \alpha \wedge \overline \alpha$ and
$\Omega = \frac{1}{2} i_{\overline \alpha}  \Psi$
provide an SU(2) structure.

Now a rotation $k_{||} \eta_+ + k_{\perp}\chi_+$ of the two orthogonal spinors $\eta_+$ and $\chi_+$, where
$k_ {||}= \cos (\phi)$ and $k_{\perp}= \sin (\phi)$, $\phi \in [0, \frac{\pi}{2}]$,
gives rise to the family of SU(2) structures $(\alpha, \tilde \omega_{\phi}, \tilde \Omega_{\phi})$ on $M$ given by
\begin{equation} \label{familysu(2)structures}
\begin{array}{l}
\tilde \omega_{\phi} = \cos (2 \phi)\, \omega +  \sin(2 \phi) Re (\Omega),\\[6pt]
\tilde \Omega_{\phi}  = - \sin (2 \phi)\, \omega +\cos (2 \phi) Re
(\Omega) + i Im (\Omega).
\end{array}
\end{equation}

\begin{definition}\cite{And} {\rm The SU(2) structure defined by
\eqref{familysu(2)structures} is called} intermediate
{\rm if $k_{||}$ and $k_{\perp}$ are both different from zero.}
\end{definition}

By working in the projection (eigen)basis, Andriot obtained in \cite{And} the SUSY equations of type IIA for intermediate SU(2) structures
in the following form:
\begin{equation}\label{SUSYeqIIA}
\left \{ \begin{array}{l}
d (Re (\alpha)) = 0,\\[6pt]
d(Re(\Omega)_{\perp}) = k_{||} k_{\perp} Re (\alpha)\wedge d(Im(\alpha)),\\[6 pt]
d (Im (\Omega)) \wedge Re (\alpha) = -d(Im (\alpha)\wedge Re (\Omega)_{||}),
\end{array} \right.
\end{equation}
where $Re( \Omega)_{\|} = \frac{1}{2} ((1 - \cos(2 \phi)) Re(\Omega) + \sin(2 \phi) \omega)$
and $Re(\Omega)_{\perp} = \frac{1}{2} ((1 +\cos(2 \phi)) Re(\Omega) - \sin(2 \phi) \omega)$.

Actually, the SUSY equations of type IIA consist of equations \eqref{SUSYeqIIA} together with the corresponding fluxes $F_0, F_2, F_4$
and $H$, but they can be obtained explicitly from these equations (see \cite{And} for details).

Next we recall the precise relationship between solutions of \eqref{SUSYeqIIA} and
the symplectic half-flat condition. In the following result, by a {\em symplectic half-flat} SU(2)
{\em structure} $(\alpha, \omega, \Omega)$ we mean that the associated
SU(3) structure given by \eqref{asocSU(3)} is symplectic half-flat.

\begin{theorem}\label{IIA}\cite[Theorem 3.1]{FU}
If $(M,\alpha, \omega, \Omega)$ is a $6$-dimensional manifold endowed with
an SU(2) structure such that the forms $\alpha$, $Re(\Omega)_{||},
Re(\Omega)_{\perp}$ and $Im(\Omega)$ satisfy the
equations~\eqref{SUSYeqIIA}, then $M$ admits a symplectic
half-flat SU(2) structure $(\hat \alpha, \hat \omega, \hat \Omega)$ with $d (Re (\hat \alpha)) =0$.
Conversely, if $M$  has a symplectic half-flat SU(2) structure
$(\hat \alpha, \hat \omega, \hat \Omega)$ such
that $d(Re( \hat \alpha)) =0$, then the forms $\alpha, Re (\Omega)_{||},
Re (\Omega)_{\perp}$ and $Im (\Omega)$ defined by
$$
 \frac{1} {k_{\perp}} Re (\Omega)_{\perp} =\hat \omega, \quad  Im (\Omega) - i \frac{1} {k_{||}} Re (\Omega)_{||} =  \hat
\Omega, \quad   Re(\alpha) + i k_{||} Im (\alpha) = \hat  \alpha
$$
provide a solution of the equations \eqref{SUSYeqIIA}.
\end{theorem}

Particular examples on compact solvmanifolds corresponding to
$\frg_{5,1}\!\oplus\R$, $A_{5,7}^{-1,-1,1}\!\oplus\R$ and $\frg_{6,118}^{0,-1,-1}$
are given in \cite{FU}.
Other solvable Lie algebras
related to the SUSY equations appear in \cite{And,AGMP,GMPT}.
As a consequence of Theorem~\ref{clasif} we obtain the complete list of compact solvmanifolds providing solutions to the
equations~\eqref{SUSYeqIIA}.

\begin{proposition}\label{IIA-solutions}
Any $6$-dimensional compact solvmanifold admitting an invariant symplectic half-flat structure
also admits a solution of the SUSY equations of type~IIA.
\end{proposition}

\begin{proof}
From Theorem~\ref{IIA} it is sufficient to show that the solvmanifolds corresponding to the Lie algebras
listed in Theorem~\ref{clasif}
admit an SU(2) structure $(\alpha,\omega,\Omega)$ with $Re(\alpha)$ a closed form and such that the
SU(3) structure $(F = \omega + \frac{i}{2} \alpha \wedge \overline \alpha,
\Psi=\alpha \wedge \Omega)$ is symplectic half-flat.
We give an explicit example for each case:

For $\fre(1,1)\oplus\fre(1,1)$ we consider the SU(2) structure $(\alpha,\omega,\Omega)$ given by
$\alpha=e^{1}+i\,e^{4}$, $\omega=e^{23}+2e^{56}$ and $\Omega=(e^{2}+i\,e^{3})\wedge((e^{5}-e^{6})+i(e^{5}+e^{6}))$.

For $\frg_{5,1}\!\oplus\R$ we consider the SU(2) structure $(\alpha,\omega,\Omega)$ given by
$\alpha=e^{1}+i\,e^{4}$, $\omega=e^{25}-e^{36}$ and $\Omega=(e^{2}+i\,e^{5})\wedge(-e^{3}+i\,e^{6})$.

For $A_{5,7}^{-1,-1,1}\oplus\R$ we consider the SU(2) structure $(\alpha,\omega,\Omega)$ given by
$\alpha=e^{5}+i\,e^{6}$, $\omega=-e^{13}+e^{24}$ and $\Omega=(e^{3}+i\,e^{1})\wedge(e^{2}+i\,e^{4})$.

For $A_{5,17}^{\alpha,-\alpha,1}\oplus\R$ we consider the SU(2) structure $(\alpha,\omega,\Omega)$ given by
$\alpha=e^{5}+i\,e^{6}$, $\omega=e^{13}+e^{24}$ and $\Omega=(e^{1}+i\,e^{3})\wedge(e^{2}+i\,e^{4})$.

For the nilpotent Lie algebra
$\frg_{6,N3}$, since $e^1$ is closed, we can consider the SU(2) structure $(\alpha,\omega,\Omega)$ given by
$\alpha=e^{1}+i\,e^{6}$, $\omega=-2e^{34}-e^{25}$ and $\Omega=(2e^{4}+i\,e^{3})\wedge(e^{5}+i\,e^{2})$.

For $\frg_{6,38}^{0}$, since $e^6$ is closed, we can consider the SU(2) structure $(\alpha,\omega,\Omega)$ given by
$\alpha=e^{6}-2i\,e^{1}$, $\omega=e^{34}+e^{52}$ and $\Omega=(e^{3}+i\,e^{4})\wedge(e^{5}+i\,e^{2})$.

For $\frg_{6,54}^{0,-1}$ we consider the SU(2) structure $(\alpha,\omega,\Omega)$ given by
$\alpha=e^{5}+i\,e^{6}$, $\omega=e^{14}+e^{23}$ and $\Omega=(e^{1}+i\,e^{4})\wedge(e^{2}+i\,e^{3})$.

Finally, for the Lie algebra $\frg_{6,118}^{0,-1,-1}$ we consider the structure $(\alpha,\omega,\Omega)$ given by
$\alpha=e^{6}+i\,e^{5}$, $\omega=e^{14}+e^{23}$ and $\Omega=(e^{1}+i\,e^{4})\wedge(e^{2}+i\,e^{3})$.
\end{proof}

\section{Appendix}\label{apendice}
In this appendix we summarize the results obtained in the before sections. In Table~1,
we show the structure equations of all {\em indecomposable unimodular}
solvable Lie algebras of dimension $6$
having symplectic forms, obtained in \cite{M}.
Also in Table~2  we show
the structure equations of {\em indecomposable non-unimodular} solvable Lie algebras
with $5$-dimensional nilradical admitting symplectic forms
according the proof of Proposition $3.2$,
while in Table~3 we consider the  {\em indecomposable non-unimodular}
solvable Lie algebras with $4$-dimensional nilradical
having both half-flat structure and symplectic forms,
but only two of those Lie algebras have symplectic half-flat structure.
Moreover, we show the structure equations of all decomposable (non-nilpotent)
solvable Lie algebras of dimension $6$ which admit both symplectic and half-flat structures (Tables~4, 5 and~6).

\vskip2cm

\centerline{
\begin{tabular}{|c|c|c|}
\hline
$\frg$& \bf{str. equations}& \bf{symplectic half-flat str.}\\
\hline
$\frg_{6,3}^{0,-1}$&$(e^{26},e^{36},0,e^{46},-e^{56},0)$&$-$\\\hline
$\frg_{6,10}^{0,0}$&$(e^{26},e^{36},0,e^{56},-e^{46},0)$&$-$\\\hline
$\frg_{6,13}^{-1,\frac{1}{2},0}$&$(-\frac{1}{2}e^{16}+e^{23},-e^{26},\frac{1}{2}e^{36},e^{46},0,0)$&$-$\\\hline
$\frg_{6,13}^{\frac{1}{2},-1,0}$&$(-\frac{1}{2}e^{16}+e^{23},\frac{1}{2}e^{26},-e^{36},e^{46},0,0)$&$-$\\\hline
$\frg_{6,15}^{-1}$&$(e^{23},e^{26},-e^{36},e^{26}+e^{46},e^{36}-e^{56},0)$&$-$\\\hline
$\frg_{6,18}^{-1,-1}$&$(e^{23},-e^{26},e^{36},e^{36}+e^{46},-e^{56},0)$&$-$\\\hline
$\frg_{6,21}^0$&$(e^{23},0,e^{26},e^{46},-e^{56},0)$&$-$\\\hline
$\frg_{6,36}^{0,0}$&$(e^{23},0,e^{26},-e^{56},e^{46},0)$&$-$\\\hline
$\frg_{6,38}^{0}$&$(e^{23},-e^{36},e^{26},e^{26}-e^{56},e^{36}+e^{46},0)$& $F=-2e^{16}+e^{34}-e^{25}$\\
&&$\rho=-2e^{135}-2e^{124}+e^{236}-e^{456}$\\\hline
$\frg_{6,54}^{0,-1}$&$(e^{16}+e^{35},-e^{26}+e^{45},e^{36},-e^{46},0,0)$& $F=e^{14}+e^{23}+e^{56}$\\
&&$\rho=e^{125}-e^{136}+e^{246}+e^{345}$\\
\hline
$\frg_{6,70}^{0,0}$&$(-e^{26}+e^{35},e^{16}+e^{45},-e^{46},e^{36},0,0)$&$-$\\\hline
$\frg_{6,78}$&$(-e^{16}+e^{25},e^{45},e^{24}+e^{36}+e^{46},e^{46},-e^{56},0)$&$-$\\
\hline
$\frg_{6,118}^{0,-1,-1}$&$(-e^{16}+e^{25},-e^{15}-e^{26},e^{36}-e^{45},e^{35}+e^{46},0,0)$& $F=e^{14}+e^{23}-e^{56}$\\
&&$\rho=e^{126}-e^{135}+e^{245}+e^{346}$\\
\hline
$\frn_{6,84}^{\pm 1}$&$(-e^{45},-e^{15}-e^{36},-e^{14}+e^{26} \mp e^{56},e^{56},-e^{46},0)$&$-$\\
\hline
\end{tabular}
}

\bigskip

\centerline{{\bf Table 1.} Indecomposable unimodular solvable (non-nilpotent) Lie algebras}

\centerline{admitting symplectic structures \cite{M}.}

\medskip

\vfill\eject

\centerline{
\begin{tabular}{|c|c|c|}\hline
$\mathfrak{g}$ & \bf{str. equations} & \bf{symplectic half-flat str.}\\\hline
$A_{6,13}^{-\frac{2}{3},\frac{1}{3},-1}$ & $(-\frac{1}{3}e^{16}+e^{23}, -\frac{2}{3}e^{26}, \frac{1}{3}e^{36}, e^{46}, -e^{56}, 0)$ & $F=-2e^{16}+e^{34}+e^{52}$\\
&&$\rho=-2e^{135}-2e^{124}-e^{356}+e^{246}$\\\hline
$A_{6,39}^{\frac{3}{2},-\frac{3}{2}}$ & $(-\frac{1}{2}e^{16}+e^{45},e^{15}+\frac{1}{2}e^{26},\frac{3}{2}e^{36},-\frac{3}{2}e^{46},e^{56},0)$ & $-$\\\hline
$A_{6,39}^{1,-1}$ & $(e^{45},e^{15}+e^{26},e^{36},-e^{46},e^{56},0)$ & $-$\\\hline
$A_{6,42}^{-1}$ & $(e^{45},e^{15}+e^{26},e^{36}+e^{56},-e^{46},e^{56},0)$ & $-$\\\hline
$A_{6,51}^{\pm 1}$ & $(e^{45},e^{15} \pm e^{46},e^{36},0,0,0)$ & $-$\\\hline
$A_{6,54}^{-1,-2}$ & $(e^{16}+e^{35},-2e^{26}+e^{45},2e^{36},-e^{46},-e^{56},0)$ &$-$\\\hline
$A_{6,54}^{\alpha,\alpha-1}$ & $(e^{16}+e^{35},(\alpha\!-\!1)e^{26}+e^{45},(1\!-\!\alpha)e^{36},-e^{46}, \alpha e^{56},0)$&$-$\\
$0\!<\!\alpha\!<\!2$& &\\\hline
$A_{6,54}^{2,1}$ & $(e^{16}+e^{35},e^{26}+e^{45},-e^{36},-e^{46}, 2e^{56},0)$&$F=e^{31}+e^{42}+2e^{65}$\\
& &$\rho=e^{346}+e^{235}-e^{145}-2e^{126}$\\\hline
$A_{6,56}^1$ & $(e^{16}+e^{35}, e^{36}+e^{45},0,-e^{46},e^{56},0)$ & $-$\\\hline
$A_{6,65}^{1,2}$ & $(e^{16}+e^{35}, e^{16}+e^{26}+e^{45},-e^{36},e^{36}-e^{46},2e^{56},0)$ & $-$ \\\hline
$A_{6,70}^{\alpha,\frac{\alpha}{2}}$ & $(\frac{\alpha}{2} e^{16}-e^{26}+e^{35}, e^{16}+\frac{\alpha}{2}e^{26}+e^{45},$ & $F=e^{13}+e^{24}-\alpha e^{65}$\\
$\alpha\not=0$&$-\frac{\alpha}{2}e^{36}-e^{46},e^{36}-\frac{\alpha}{2}e^{46},\alpha e^{56},0)$&$\rho=-\alpha e^{126}-e^{145}+e^{235}+\alpha e^{346}$ \\\hline
$A_{6,71}^{-\frac{3}{2}}$ & $(\frac{3}{2}e^{16}+e^{25}, \frac{1}{2}e^{26}+e^{35},-\frac{1}{2}e^{36}+e^{45},-\frac{3}{2}e^{46},e^{56},0)$ & $F=e^{41}+e^{23}+2e^{56}$ \\
&&$\rho=-e^{245}+2e^{346}-2e^{126}-e^{135}$\\\hline
$A_{6,76}^{-3}$ & $(-5e^{16}+e^{25}, -2e^{26}+e^{45},e^{24}-e^{36},e^{46},-3e^{56},0)$ & $-$ \\\hline
$A_{6,82}^{2,5,9}$ & $(2e^{16}+e^{24}+e^{35}, 6e^{26},10e^{36},-4e^{46},-8e^{56},0)$ & $-$ \\\hline
$A_{6,94}^{-3}$ & $(-e^{16}+e^{25}+e^{34},-2e^{26}+e^{35},-3e^{36},2e^{46},e^{56},0)$ & $-$ \\\hline
$A_{6,94}^{-\frac{5}{3}}$ & $(\frac{1}{3}e^{16}+e^{25}+e^{34},-\frac{2}{3}e^{26}+e^{35},-\frac{5}{3}e^{36},2e^{46},e^{56},0)$ & $-$ \\\hline
$A_{6,94}^{-1}$ & $(e^{16}+e^{25}+e^{34},e^{35},-e^{36},2e^{46},e^{56},0)$ & $-$ \\\hline
\end{tabular}
}

\bigskip

\centerline{{\bf Table 2.} Indecomposable non-unimodular solvable Lie algebras with 5-dimensional nilradical}

\centerline{admitting both symplectic and half-flat structure \cite{FS2}.}

\medskip

\vfill\eject

\small{
\centerline{
\begin{tabular}{|c|c|c|}
\hline
$\mathfrak{g}$ & \bf{str. equations} & \bf{symplectic half-flat str.}\\\hline
$N_{6,1}^{\alpha,\beta,-\alpha,-\beta}\, \alpha\beta\neq0$&$(\alpha e^{15}+\beta e^{16},-\alpha e^{25}-\beta e^{26},e^{36},e^{45},0,0)$&$-$\\\hline
$N_{6,1}^{\alpha,\beta,0,-1}\, \alpha\beta\neq 0$&$(\alpha e^{15}+\beta e^{16}, -e^{26},e^{36},e^{45},0,0)$&$-$\\\hline
$N_{6,1}^{\alpha,\beta,-1,0}\, \alpha\beta\neq 0$&$(\alpha e^{15}+\beta e^{16}, -e^{25},e^{36},e^{45},0,0)$&$-$\\\hline
$N_{6,2}^{-1,\beta,-\beta}$&$(-e^{15}+\beta e^{16},e^{25}-\beta e^{26},e^{36},e^{35}+e^{46},0,0)$&$-$\\\hline
$N_{6,2}^{0,-1,\gamma}$&$(-e^{16},e^{25}+\gamma e^{26},e^{36},e^{35}+e^{46},0,0)$&$-$\\\hline
$N_{6,7}^{0,\beta,0}\ \beta\neq 0$ &$(-e^{26},e^{16},e^{35},e^{35}+\beta e^{36}+e^{45},0,0)$&$-$\\\hline
$N_{6,13}^{\alpha,\beta,-\alpha,-\beta}$ &$(\alpha e^{15}+\beta e^{16},-\alpha e^{25}-\beta e^{26},e^{36}-e^{45},$&$-$\\
$\alpha^2+\beta^2\neq 0$&$e^{35}+e^{46},0,0)$&\\
$(\alpha,\beta)\neq(0,\pm2)$&&\\\hline
$N_{6,13}^{0,-2,0,2}$ &$(-2 e^{16}, 2 e^{26},e^{36}-e^{45},e^{35}+e^{46},0,0)$&$F=e^{12}+e^{35}+e^{46}$\\
&&$\rho=e^{134}-e^{156}-e^{236}+e^{245}$\\\hline
$N_{6,14}^{\alpha,\beta,0}\ \alpha\beta\neq 0 $&$(\alpha e^{15}+\beta e^{16},e^{26},-e^{45},e^{35},0,0)$&$-$\\\hline
$N_{6,15}^{0,\beta,\gamma,0}\ \beta\neq 0$&$(e^{15}+\gamma e^{16}-e^{26},e^{16}+e^{25}+\gamma e^{26},$&$-$\\
&$-\beta e^{45},\beta e^{35},0,0)$&\\\hline
$N_{6,16}^{0,0}$&$(e^{16},e^{15}+e^{26},-e^{45},e^{35},0,0)$&$-$\\\hline
$N_{6,17}^{0}$&$(0,e^{15},e^{36}-e^{45},e^{35}+e^{46},0,0)$&$-$\\\hline
$N_{6,18}^{0,\beta,0}\ \beta\neq 0$&$(e^{16}-e^{25},e^{15}+e^{26},-\beta e^{45},\beta e^{35},0,0)$&$-$\\\hline
$N_{6,20}^{0,-1}$&$(-e^{56},-e^{26},e^{36},e^{45},0,0)$&$-$\\\hline
$N_{6,22}^{\alpha,0}\ \alpha \neq0$&$(e^{15}+\alpha e^{16},e^{26},0,e^{35},0,0)$&$-$\\\hline
$N_{6,23}^{\alpha,0}$&$(e^{15}-e^{26},e^{16}+e^{25},0,e^{35}+\alpha e^{36},0,0)$&$-$\\\hline
$N_{6,26}^{0}$&$(-e^{56},e^{26},-e^{45},e^{35},0,0)$&$-$\\\hline
$N_{6,28}$&$(-e^{24}+e^{15},-e^{34}+e^{26}, -e^{35}+2e^{36},$&$-$\\
&$e^{45}-e^{46},0,0)$&\\\hline
$N_{6,29}^{\alpha,\beta}$&$(-e^{23}+e^{15}+e^{16},e^{25},e^{36},\alpha e^{45}+\beta e^{46},0,0)$&$-$\\
$\alpha^2+\beta^2\neq0$&&\\\hline
$N_{6,30}^{\alpha}$&$(-e^{23}+2e^{15},e^{25},e^{26}+e^{35},\alpha e^{45}+e^{46},0,0)$&$-$\\\hline
$N_{6,32}^{\alpha}$&$(-e^{23}+e^{45}+e^{16},e^{25}+\alpha e^{26},$&$-$\\
&$(1-\alpha)e^{36}-e^{35},e^{46},0,0)$&\\\hline
$N_{6,33}$&$(-e^{23}+e^{15}+e^{16},e^{25}, e^{36},e^{36}+e^{46},0,0)$&$-$\\\hline
$N_{6,34}^{\alpha}$&$(-e^{23}+e^{15}+(1+\alpha)e^{16},e^{25}+\alpha e^{26},$&$-$\\
&$e^{36},e^{35}+e^{46},0,0)$&\\\hline
$N_{6,35}^{\alpha,\beta}\ \alpha \neq 0$&$(-e^{23}+2e^{16},-e^{35}+e^{26},e^{36}+e^{25},$&$-$\\
&$\alpha e^{45}+\beta e^{46},0,0)$&\\\hline
$N_{6,37}^{\alpha}$&$(-e^{23}+e^{45}+2e^{16},e^{26}-e^{35}-\alpha e^{36},$&$-$\\
&$e^{25}+\alpha e^{26}+e^{36},2e^{46},0,0)$&\\\hline
$N_{6,38}$&$(-e^{23}+e^{15}+e^{16},e^{25}, e^{36},-e^{56},0,0)$&$-$\\\hline
$N_{6,39}$&$(-e^{23}+2e^{16},-e^{35}+e^{26}, e^{25}+e^{36},-e^{56},0,0)$&$-$\\\hline
\end{tabular}
}

\bigskip

\centerline{{\bf Table 3.} Indecomposable non-unimodular symplectic solvable Lie algebras}

\centerline{with 4-dimensional nilradical.}


\vfill\eject

\centerline{
\begin{tabular}{|c|c|c|}
\hline
$\frg$& \bf{str. equations}& \bf{symplectic half-flat str.}\\
\hline
$\fre(2)\oplus\fre(2)$&$(0,-e^{13},e^{12},0,-e^{46},e^{45})$&$-$\\\hline
$\fre(1,1)\oplus\fre(1,1)$&$(0,-e^{13},-e^{12},0,-e^{46},-e^{45})$&$F=e^{14}+e^{23}+2e^{56}$\\
&&$\rho=e^{125}-e^{126}-e^{135}-e^{136}$\\
&&$+e^{245}+e^{246}+e^{345}-e^{346}$\\
\hline
$\fre(2)\oplus\R^3$&$(0,-e^{13},e^{12},0,0,0)$&$-$\\\hline
$\fre(1,1)\oplus\R^3$&$(0,-e^{13},-e^{12},0,0,0)$&$-$\\\hline
$\fre(2)\oplus\fre(1,1)$&$(0,-e^{13},e^{12},0,-e^{46},-e^{45})$&$-$\\\hline
$\fre(2)\oplus\frh$&$(0,-e^{13},e^{12},0,0,e^{45})$&$-$\\\hline
$\fre(1,1)\oplus\frh$&$(0,-e^{13},-e^{12},0,0,e^{45})$&$-$\\\hline
$\fre(2)\oplus\frr_2\oplus\R$&$(0,-e^{13},e^{12},0,-e^{45},0)$&$-$\\\hline
$\fre(1,1)\oplus\frr_2\oplus\R$&$(0,-e^{13},-e^{12},0,-e^{45},0)$&$-$\\\hline
\end{tabular}
}

\medskip

\centerline{{\bf Table 4.} $3\oplus3$ decomposable (non-nilpotent) Lie algebras}

\centerline{admitting both symplectic and half-flat structures.}

\medskip

\bigskip

\centerline{
\begin{tabular}{|c|c|c|}
\hline
$\frg$& \bf{str. equations}& \bf{symplectic half-flat str.}\\
\hline
$A_{4,1}\oplus\frr_2$&$(e^{24},e^{34},0,0,0,e^{56})$&$-$\\\hline
$A_{4,9}^{-\frac{1}{2}}\oplus\frr_2$&$(\frac{1}{2}e^{14}+e^{23},e^{24},-\frac{1}{2}e^{34},0,0,e^{56})$&$-$\\\hline
$A_{4,12}\oplus\frr_2$&$(e^{13}+e^{24},-e^{14}+e^{23},0,0,0,e^{56})$&$-$\\\hline
$\frr_2\oplus\frr_2\oplus\frr_2$&$(0,-e^{12},0,-e^{34},0,-e^{56})$&$-$\\\hline
\end{tabular}
}

\medskip

\centerline{{\bf Table 5.} $4\oplus2$ decomposable (non-nilpotent) Lie algebras}

\centerline{admitting both symplectic and half-flat structures.}

\medskip

\bigskip

\centerline{
\begin{tabular}{|c|c|c|}
\hline
$\frg$& \bf{str. equations}& \bf{symplectic half-flat str.}\\
\hline
$A_{5,7}^{-1,\beta,-\beta}\oplus\R$&$(e^{15},-e^{25},\beta e^{35},-\beta e^{45},0,0)$&$-$\\
$0<\beta<1$&&\\\hline
$A_{5,7}^{-1,-1,1}\oplus\R$&$(e^{15},-e^{25},-e^{35},e^{45},0,0)$&$F=-e^{13}+e^{24}+e^{56}$\\
&&$\rho=-e^{126}-e^{145}-e^{235}-e^{346}$\\
\hline
$A_{5,8}^{-1}\oplus\R$&$(e^{25},0,e^{35},-e^{45},0,0)$&$-$\\\hline
$A_{5,13}^{-1,0,\gamma}\oplus\R$&$(e^{15},-e^{25},\gamma e^{45},-\gamma e^{35},0,0)$&$-$\\
$\gamma>0$&&\\\hline
$A_{5,14}^0\oplus\R$&$(e^{25},0,e^{45},-e^{35},0,0)$&$-$\\\hline
$A_{5,15}^{-1}\oplus\R$&$(e^{15}+e^{25},e^{25},-e^{35}+e^{45},-e^{45},0,0)$&$-$\\\hline
$A_{5,17}^{0,0,\gamma}\oplus\R$&$(e^{25},-e^{15},\gamma e^{45},-\gamma e^{35},0,0)$&$-$\\
$0<\gamma<1$&&\\\hline
$A_{5,17}^{\alpha,-\alpha,1}\oplus\R$&$(\alpha e^{15}+e^{25},-e^{15}+\alpha e^{25},$&$F=e^{13}+e^{24}+e^{56}$\\
$\alpha\geq 0$&$-\alpha e^{35}+e^{45},-e^{35}-\alpha e^{45},0,0)$&$\rho=e^{125}-e^{146}+e^{236}-e^{345}$\\
\hline
$A_{5,18}^0\oplus\R$&$(e^{25}+e^{35},-e^{15}+e^{45},e^{45},-e^{35},0,0)$&$-$\\\hline
$A_{5,19}^{-1,2}\oplus\R$&$(-e^{15}+e^{23},e^{25},-2 e^{35},2e^{45},0,0)$&$-$\\\hline
$A_{5,36}\oplus\R$&$(e^{14}+e^{23},e^{24}-e^{25},e^{35},0,0,0)$&$-$\\\hline
$A_{5,37}\oplus\R$&$(2e^{14}+e^{23},e^{24}+e^{35},-e^{25}+e^{34},0,0,0)$&$-$\\\hline
\end{tabular}
}
\medskip

\centerline{{\bf Table 6.} $5\oplus1$ decomposable (non-nilpotent) Lie algebras}

\centerline{admitting both symplectic and half-flat structures.}

\medskip

\vskip.5cm

\noindent {\bf Acknowledgments.}
We would like to thank Sergio Console and Anna Fino for many useful comments and suggestions.
We also thank the anonymous referee
for useful advice which improved the presentation of the paper.
This work has been partially supported through  Projects MICINN (Spain) MTM2008-06540-C01/02
and MEC (Spain) MTM2011-28326-C01/02.

\smallskip

{\small


\begin{thebibliography}{33}

\bibitem{And} D. Andriot, New supersymmetric flux vacua with intermediate $SU(2)$-structure, \emph{J. High Energy Phys.\/} {\bf 0808}, 096 (2008).

\bibitem{AGMP} D. Andriot, E. Goi, R. Minasian, M. Petrini,
Supersymmetry breaking branes on solvmanifolds and de Sitter vacua in string theory, \emph{J. High Energy Phys.\/} {\bf 1105}, 028 (2011).

\bibitem{B} C. Bock, On Low-Dimensional Solvmanifolds, arXiv:0903.2926v4 [math.DG].

\bibitem{C} D. Conti, Half-flat nilmanifolds, \emph{Math. Ann.\/} {\bf 350} (2011), no. 1, 155-–168.

\bibitem{CF} D. Conti, M. Fern\'andez, Nilmanifolds with a calibrated $G_2$-structure, \emph{Differ. Geom. Appl.\/} {\bf 29} (2011) 493–-506.

\bibitem{CT} D. Conti, A. Tomassini, Special symplectic six-manifolds, \emph{Q. J. Math.\/} {\bf 58} (2007) 297–-311.

\bibitem{VTSS} V. Cort\'es, T. Leistner, L. Sch\"afer, F. Schulte-Hengesbach,
Half-flat structures and special holonomy, \emph{Proc. Lond. Math. Soc.\/} (3) {\bf 102} (2011), no. 1, 113-–158.

\bibitem{FernandezGray} M.~Fern\'andez, A.~Gray, Riemannian manifolds with structure group $G_2$,
\emph{Annali di Mat. Pura Appl.\/} {\bf  32} (1982), 19--45.

\bibitem{FLS} M. Fern\'andez, M. de Le\'on, M. Saralegui, A six dimensional symplectic solvmanifold without K\"ahler structures,
\emph{Osaka J. Math.\/} {\bf 33} (1996), 19--35.

\bibitem{FU} A. Fino, L. Ugarte, On the geometry underlying supersymmetric flux vacua with intermediate SU(2) structure,
\emph{Class. Quantum Grav.\/} {\bf 28} (2011), no. 7, 075004, 21 pp.

\bibitem{FS1} M. Freibert, F. Schulte-Hengesbach, Half-flat structures on decomposable Lie groups,
\emph{Transform. Groups\/} {\bf 17} (2012), no. 1, 123--141.

\bibitem{FS2} M. Freibert, F. Schulte-Hengesbach, Half-flat structures on indecomposable Lie groups, to appear in
\emph{Transform. Groups\/}.

\bibitem{Gor} V.V. Gorbatsevich, Symplectic structures and cohomologies on some
solv-manifolds, \emph{Siberian Math. J.\/} {\bf 44} (2003), no. 2, 260--274.

\bibitem{GMPT} M. Gra\~na, R. Minasian, M. Petrini, A. Tomasiello, A scan for new $N=1$ vacua on twisted tori,
\emph{J. High Energy Phys.\/} {\bf 0705}, 031 (2007).

\bibitem{HL} R.~Harvey and H.B.~Lawson, Calibrated geometries,
\emph{Acta Math.\/} 148 (3) (1982) 47--157.

\bibitem{H1} N. Hitchin, The geometry of three-forms in six dimensions, \emph{J. Differ. Geom.\/} {\bf 55} (2000) 547–-576.

\bibitem{H2} N. Hitchin, Stable forms and special metrics, Global differential geometry:
the mathematical legacy of Alfred Gray (American Mathematical Society, Providence, RI, 2001) 70--89.

\bibitem{M} M. Macr\`i, Cohomological properties of unimodular six dimensional solvable Lie algebras,
arXiv:1111.5958v2 [math.DG].

\bibitem{Mubara} G.M. Mubarakzyanov, Classification of solvable Lie algebras of sixth order
with a non-nilpotent basis element (Russian), \emph{Izv. Vyssh. Uchebn. Zaved. Mat.\/}
{\bf 35} (1963), no. 4, 104--116.

\bibitem{S} F. Schulte-Hengesbach, Half-flat structures on products of three-dimensional Lie groups,
\emph{J. Geom. Phys.\/} {\bf 60} (2010), no. 11, 1726--1740.

\bibitem{Sha} A. Shabanskaya, Classification of Six Dimensional Solvable Indecomposable Lie Algebras with a
codimension one nilradical over $\mathbb{R}$, Ph.D. Thesis, University of Toledo,  Ohio, 2011,
181 pp.

\bibitem{TV} A. Tomassini, L. Vezzoni, On symplectic half-flat manifolds,
\emph{Manuscripta Math.\/} {\bf 125} (2008), no. 4, 515–530.

\bibitem{T} A. Tralle, J. Oprea, Symplectic manifolds with no K\"ahler structures, Lectures Notes in Mathematics, 1661.
Springer-Verlag, Berlin, 1997.

\bibitem{Tu} P. Turkowski, Solvable Lie algebras of dimension
six, \emph{J. Math. Phys.\/} {\bf 31} (1990) 1344--1350.

\bibitem{Y} T. Yamada, A pseudo-K\"ahler structure on a nontoral compact complex parallelizable solvmanifold,
\emph{Geom. Dedicata} {\bf 112} (2005), 115--122.

\end{thebibliography}
\end{document}